\documentclass[12pt]{amsart}
\usepackage{amsmath,amssymb,amsthm,amsfonts,enumerate,array,color,lscape,fancyhdr,layout,pst-all}
\usepackage{xcolor}
\usepackage[english]{babel}
\usepackage{young}
\usepackage{float}
\usepackage[all]{xy}
\usepackage[normalem]{ulem}
\usepackage[active]{srcltx}
\usepackage{mathrsfs}
\usepackage{etoolbox}
\usepackage[colorlinks=true, breaklinks=true, allcolors=blue]{hyperref}

\usepackage{xspace} 
\newlength\yStones
\newlength\xStones
\newlength\xxStones

\makeatletter
\def\Stones{\pst@object{Stones}}
\def\Stones@i#1{%
  \pst@killglue%
  \begingroup%
  \use@par%
  \setlength\xxStones{\xStones}%
  \expandafter\Stones@ii#1,,\@nil
  \endgroup
  \global\addtolength\xStones{0.6cm}%
  \global\addtolength\yStones{-7.5mm}}%
\def\Stones@ii#1,#2,#3\@nil{%
  \rput(\xxStones,\yStones){%
    \psframebox[framesep=0]{%
      \parbox[c][6mm][c]{11mm}{\makebox[11mm]{$#1$}}}}%
  \addtolength\xxStones{1.2cm}%
  \ifx\relax#2\relax\else\Stones@ii#2,#3\@nil\fi}
\makeatother

\def\Stone#1{\fbox{\makebox[12mm]{\strut#1}}\kern2pt}

\newcommand{\C}{\mathbb{C}}

\newcommand{\Z}{\mathbb{Z}}

\newcommand{\A}{\mathbb{A}}

\newcommand{\g}{\mathfrak{g}}
\newcommand{\h}{\mathfrak{h}}

\newcommand{\gln}{\mathfrak{gl}_n}
\newcommand{\cV}{V}

\newcommand{\cO}{\mathcal{O}}
\newcommand{\gl}{\mathfrak{gl}}

\newcommand{\Ann}{\mathrm{Ann} \,}

\newcommand{\Spec}{\mathrm{Spec} \,}
\newcommand{\Hom}{\mathrm{Hom} \,}
\newcommand{\End}{\mathrm{End}}

\renewcommand{\k}{{\Bbbk}}
 
\newcommand{\Der}{\mathrm{Der}}
\newcommand{\cM}{\mathcal{M}}

\newcommand{\fxpartial}[2]{\frac{\partial #1}{\partial #2}}
\newcommand{\xpartial}[1]{\frac{\partial}{\partial #1}}
\newcommand{\cJ}{\mathcal{J}}
\newcommand{\cN}{\mathcal{N}}
\newcommand{\cP}{\mathcal{P}}
\newcommand{\aL}{\widehat{\mathrm{L}}}
\newcommand{\sAV}{\widehat{\mathcal{A}\mathcal{V}}}
\newcommand{\sAVd}{\widehat{\mathcal{A}\mathcal{V}}_{\mathit{diff}}}
\newcommand{\aAV}{A\cV}

\newcommand{\cD}{\mathcal{D}}
\newcommand{\cL}{\mathcal{L}}
\newcommand{\fm}{\mathfrak{m}}

\newcommand{\bk}{\Bbbk}
\newcommand{\GKdim}{\mathrm{GKdim}}
\newcommand{\rmd}{\mathrm{GKdim}}
\newcommand{\rmgr}{\mathrm{gr}\,}
\newcommand{\rmCh}{\mathrm{Ch}}
\newcommand{\aD}{D} 
\newcommand{\hcL}{\widehat{\cL}}

\newcommand{\GL}{\mathrm{GL}}
\newcommand{\fsl}{\mathfrak{sl}}

\newcommand{\glnbund}{\mathcal{U}_{gl_N}}
\newcommand{\Zglnbund}{\mathcal{Z}_{gl_N}}
\newcommand{\aLA}{\widehat{\mathrm{L}}^A}

\newtheorem{theorem}{Theorem}[section]
\newtheorem{lemma}[theorem]{Lemma}
\newtheorem{corollary}[theorem]{Corollary}
\newtheorem{proposition}[theorem]{Proposition}
\newtheorem{conjecture}[theorem]{Conjecture}

\theoremstyle{definition}
\newtheorem{example}[theorem]{Example}
\newtheorem{remark}[theorem]{Remark}

\newtheorem{definition}[theorem]{Definition}
\allowdisplaybreaks

\begin{document}
\begin{title}[Differentiable holonomic $AV$-modules]{Differentiable holonomic $AV$-modules}
\end{title}

\author[Y. Billig]{Yuly Billig}
\address{Carleton University \\ Ottawa \\ Canada}
\email{}

\author[H. Rocha]{Henrique Rocha}
\email{}


\begin{abstract}
We study differentiable holonomic sheaves of $A\cV$-modules on a smooth quasi-projective variety. We show that a simple differentiable holonomic sheaf $\cM$ of $A\cV$-modules is locally the tensor product of a simple holonomic $\aD$-module and a simple finite-dimensional $\gl_n$-module $W$. In particular, in the case when $W$ is integrable, $\cM$ is the tensor product of a simple holonomic $\cD$-module and the tensor module associated with $W$.
\end{abstract}




\maketitle


\tableofcontents    


\section*{Introduction}

The theory of $\cD$-modules has far-reaching applications in various fields, including algebraic geometry, representation theory, and mathematical physics. In essence, it offers a strong algebraic framework for working with differential operators and systems of differential equations on algebraic varieties. A $\cD$-module on a smooth algebraic variety can be viewed as a sheaf of modules over the sheaf of differential operators, encapsulating both the geometry of the variety and the action of vector fields in a coherent algebraic structure. One key feature of $\cD$-modules is that they are associated with an $\cO$-linear representation of the tangent sheaf, where $\cO$ denotes the structure sheaf of the variety.

Recently, this framework has been generalized to allow a more flexible relationship between derivations and functions. In the theory of $A\cV$-modules, the associated representation of the Lie algebra of vector fields is not assumed to be linear over the ring of regular functions. A primary example of an $A\cV$-module is a space of tensor fields on a variety. The algebraic theory of $A\cV$-modules over smooth affine varieties has been developed extensively in~\cite{BF18,BFN19,BNZ21,BIN23,BR23, BI23, XL23, BB24, BR24}.

The sheaf-theoretic version of $A\cV$-modules, which allows these structures to be defined on a not necessarily affine variety, was initiated in~\cite{BR23}, where the authors demonstrated that an $A\cV$-module that is finitely generated over the coordinate ring of an affine variety can be sheafified. This led to the construction of a sheaf of associative algebras $\sAV$, introduced in~\cite{BI23} and further studied in \cite{BB24}, which serves as a replacement for the smash product algebra in the affine case. Modules over $\sAV$, or simply $\sAV$-modules, are thus sheaf-theoretic analogs of $A\cV$-modules.

The development of the theory of $\sAV$-modules has led to a realization of the significance of the notion of differentiability. Differentiable modules form an important subcategory in the category of $\sAV$-modules. These are the modules for which the action of the tangent sheaf is a differential operator of finite order $N$, in the sense of Grothendieck. Differentiable $\sAV$-modules can be seen as a natural generalization of $\cD$-modules, which correspond to the case $N=0$. They play a central role throughout this paper, and we use this property to analyze the local structure of the main type of $\sAV$-modules we want to investigate, the holonomic $\sAV$-modules.

In this paper, we study the notion of the Gelfand-Kirillov (GK) dimension for $\sAV$-modules, which is a way to measure the growth of algebras and modules. We prove that the GK-dimension of any $\sAV$-module is bounded below by the dimension of the underlying variety. $\sAV$-modules with this minimal possible dimension are called holonomic, in analogy with holonomic $\cD$-modules. 

Our main results establish key structural properties of holonomic differentiable $\sAV$-modules. We prove that these modules have finite length. We also characterize the $\sAV$-modules
that are coherent over $\cO$, and show that these are necessarily holonomic and locally free, thus forming vector bundles on the underlying variety. We then study simple holonomic $\sAV$-modules. We prove that they are locally isomorphic to the tensor product of a simple holonomic $\aD$-module and a simple finite-dimensional $\gln$-module $W$. When $W$ is an integrable $\gln$-module, the simple holonomic $\sAV$-module is globally the tensor product of a simple holonomic $\cD$-module and a tensor module associated with $W$. If $W$ is not integrable, we can still give a global description in terms of a charged $\cD$-module and a tensor module associated with an $\fsl_n$-module. We conclude the paper by establishing a generalization of the main result in~\cite{BNZ21} and proving that a simple holonomic $\sAV$-module remains locally simple as a module over the tangent sheaf as long as it is not associated with a fundamental representation of $\gl_n$.

The paper is organized as follows. In Section~\ref{section:gkdim}, we establish the preliminary results about the Gelfand-Kirillov dimension that we will use in this paper. The sheaf of associative algebras $\sAV$ is introduced in Section~\ref{section:sheafAVdiffmodules}, where we also review the basic properties of differentiable $\sAV$-modules. Section~\ref{section:holavmodules} is focused on the general study of holonomic $\sAV$-modules. In Section~\ref{section:coherentavmod}, we investigate $\sAV$-modules that are coherent as $\cO$-modules. In the final Section~\ref{section:irreducibleholmodules}, we describe simple differentiable holonomic $\sAV$-modules.

Throughout, we work over an uncountable algebraically closed field $\bk$ of characteristic zero, and all varieties and sheaves are defined over $\bk$. Unless otherwise stated, $X$ denotes a smooth irreducible quasi-projective variety of dimension $n$, though in certain arguments we restrict to the affine case.

{\bf Acknowledgements:} 
Y.B. gratefully acknowledges support with a Discovery grant from the Natural Sciences and Engineering Research Council of Canada. H.R. thanks Carleton University for the excellent work environment. 

\section{Gelfand-Kirillov dimension}\label{section:gkdim}

Let $R$ be a unital associative algebra. A subspace $F \subset R$ is called a \emph{subframe} of $R$ if $\dim F < \infty$ and $1 \in F$. A subframe $F \subset R$ is a \emph{frame} if $F$ is a generating subset for the algebra $R$. For an $R$-module $M$, the \emph{Gelfand-Kirillov dimension} of $M$ is
\[
\GKdim_R(M) = \sup_{\substack{F \subset R \text{ subframe} \\ M_0 \subset M} } \overline{\lim_l} \log_l \dim_{\bk} F^{l}M_0,
\]
where $M_0 \subset M$ are finite-dimensional subspaces of $M$ and $\overline{\lim}$ denotes the limit superior. 
If the algebra $R$ is finitely generated by a frame $F$ and the $R$-module $M$ is generated by a finite-dimensional subspace $M_0$, then
\[
\GKdim_{R}(M) = \overline{\lim_l} \log_l \dim_{\bk} F^l M_0.
\]
 In this case, the Gelfand-Kirillov dimension does not depend on the choice of a frame or a generating subspace. 

\begin{lemma}
Let $R,\ S$ be associative algebras and let $M$ and $N$ be modules over $R$ and $S$, respectively. Then
\[
\max \left\{\GKdim_R(M), \ \GKdim_S(N) \right\} \leq \GKdim_{R\otimes S}\left (M\otimes_{\bk} N \right) \leq \GKdim_R(M) + \GKdim_S(N).
\]
\end{lemma}
\begin{proof} 
The first inequality follows from the definition of Gelfand-Kirillov dimension. The second follows from
\[
\overline{\lim} (x_l+y_l) \leq \overline{\lim} (x_l) +\overline{\lim} (y_l)
\]
and
\[
\log_l \dim (F^l M_0\otimes G^l N_0) =\log_l \dim (F^l M_0) + \log_l \dim (G^l N_0).
\]
\end{proof}

\begin{lemma}{\cite[Proposition 10.1.22]{NVO87}}\label{lemma:commalgebrarelation}
Let $M$ be finitely generated over a commutative algebra $R$, then $ \GKdim (M)= \GKdim (R/\Ann_R(M))$. 
\end{lemma}

We say that an associative algebra $R$ has a \emph{good filtration} if $R$ has a filtration 
\[
R = \bigcup_{k \geq 0} R_k
\]
such that $\rmgr R$ is a finitely generated algebra. Similarly, we say that an $R$-module $M$ has a \emph{good filtration} if there exists a family of subspaces $M_i$, $i\geq -1$, with $M_{-1}=0$, $M_{i} \subset M_{i+1}$ such that
$M = \bigcup_{i=0}^{\infty}M_i$,
$R_rM_s \subset M_{r+s}$, and 
\[
\rmgr M = \bigoplus_{i\geq 0} M_i/M_{i-1}
\]
is finitely generated as a $\rmgr R$-module. 
\begin{lemma}{\cite[Corollary 1.1]{MS89}}
Let $M$ be a module over an associative algebra $R$. Suppose that both $M$ and $R$ have good filtrations and that $\rmgr R$ is a commutative algebra without nilpotent elements. Then, $\GKdim_{\rmgr R} (\rmgr M) =\GKdim_R(M)$.
\end{lemma}

\section{The sheaf $\sAV$ and differentiable modules}
\label{section:sheafAVdiffmodules}

For a commutative algebra $B$, we denote the structure sheaf of $B$ by $\cO_B$. For each $f \in B$, $f\neq 0$, we denote by $D(f) = \{ \mathfrak{p} \in \Spec(B) \mid f \notin \mathfrak{p} \}$ a basic open set. If $M$ is a $B$-module, we denote by $\widetilde{M}$ the sheaf of $\cO_B$-modules associated to $M$ on $\Spec(B)$. Explicitly, $\widetilde{M}(D(f)) = B_f \otimes_B M$ is the localization of $M$ by the multiplicative set $\{f^k \mid k\geq 0\}$. Let $X$ be a scheme with structure sheaf $\cO_X$. A sheaf $\cM$ of $\cO_X$-modules is called \emph{quasi-coherent} if there exists an affine open cover $\{U_{\alpha} \}$ of $X$ and modules $M_{\alpha}$ over $\cO_X(U_{\alpha})$ such that $\cM|_{U_{\alpha}} = \widetilde{M_{\alpha}}$. If each $M_{\alpha}$ is finitely generated over $\cO_X(U_{\alpha})$, $\cM$ is called \emph{coherent}. For an introduction to varieties and sheaves, we refer to~\cite{Mum99}. 

If $\g$ is a Lie algebra, then $U(\g)$ is a Hopf algebra. If $B$ is a commutative algebra on which $\g$ acts by derivations, then the smash product $B \# U(\g)$ is an associative algebra defined on the vector space $B \otimes U(\g)$ with the product
\[
(f \otimes u)(g \otimes v)= \sum_{i} f u_i^{(1)}(g)  \otimes u_i^{(2)}v,
\]
where $\displaystyle u \mapsto \sum_{i} u_i^{(1)} \otimes u_i^{(2)}$ is the coproduct. We denote the element $f \otimes u \in B \# U(\g)$ by $f\# u$. 
One can easily check that $B \# \g$ is a Lie subalgebra in $B \# U(\g)$.
For more details about Hopf algebras and smash products, we refer to~\cite{DNR01}.

Suppose that $X$ is an affine variety with coordinate ring $A$ and $V = \Der(A)$. An \emph{$A\cV$-module} on $X$ is a module over the associative algebra $A \#U(V)$. The category of $A\cV$-modules is closed under tensor products $M_1 \otimes_A M_2$ and taking dual modules $M^{\ast} =\Hom_{A}(M,A)$, see \cite{BFN19}. Furthermore, for a $\cV$-module $W$, we can define an induced $A\cV$-module 
\[
M = A \# U(\cV) \otimes_{U(\cV)} W.
\]
In this case, $M \cong A \otimes_{\bk} W$ as a vector space, the $A$-module action is given by multiplication in the left factor and
\[
\eta(f \otimes w) = \eta(f) \otimes w + f \otimes \eta w, \ \eta \in \cV, \ f \in A, \ w \in W.
\]

 In this paper, $X$ is a smooth irreducible quasi-projective algebraic variety of dimension $n$ with structure sheaf $\cO$, and tangent sheaf $\Theta$. A \emph{sheaf of $A\cV$-modules} on $X$ is a quasi-coherent sheaf $\cM$ such that $M=\cM(U)$ is an $A\cV$-module on $U$ for each affine open set $U\subset X$ such that $\cM|_U = \widetilde{M}$.

Sheaves of $A\cV$-modules as defined above were introduced in~\cite{BR23}, where the authors showed that if an $A\cV$-module $M$ is finitely generated as a module over the coordinate ring of an affine variety $X$, then $\widetilde{M}$ is a sheaf of $A\cV$-modules.

We want to study a geometric version of the theory of $A\cV$-modules with the assistance of the topologically quasi-coherent sheaf $\sAV$ of associative algebras constructed in~\cite{BI23} and elaborated in~\cite{BB24}. To provide a local description of the sheaf $\sAV$, we require the concept of étale charts. 

\begin{definition}
Let $U \subset X$ be an affine open set. We say that $U$ is an \emph{étale chart} if there exist $x_1,\dots,x_n \in \cO(U)$, called \emph{uniformizing parameters}, such that
\begin{enumerate}
\item the set $\{ x_1,\dots,x_n \}$ is algebraically independent;
\item every $f \in \cO(U)$ is algebraic over $\bk[x_1,\dots,x_n]$;
\item derivations $\xpartial{x_1},\dots,\xpartial{x_n}$ extend to a derivation of $\cO(U)$.
\end{enumerate}
\end{definition}

Since $\cO(U)$ is algebraic over $\bk[x_1,\dots,x_n]$, the extensions of $\xpartial{x_1},\dots,\xpartial{x_n}$ are unique. Additionally, $\Der(\cO(U))$ is a free module over $\cO(U)$ with basis $\partial_{x_1} = \xpartial{x_1},\dots,\partial_{x_n} = \xpartial{x_n}$. Any smooth irreducible quasi-projective variety $X$ admits a finite open cover of étale charts.

Let 
\[
L^A_+ = \fm \Der_{A}(A[X_1,\dots,X_n]) = \bigoplus_{i=1}^n \fm \xpartial{X_i}
\]
be the Lie algebra of $A$-derivations of the $A$-algebra $A[X_1,\dots,X_n]$ with values in the ideal $\fm$ of $A[X_1,\dots,X_n]$ generated by $X_1,\dots,X_n$. When $A = \k$, we drop the superscript.
This Lie algebra $\Z$-graded, 
\[
L^A_+ = \mathop\oplus_{k=0}^\infty L^A_k
\]
and has a decreasing filtration by subspaces $L^A_{\geq s} = \mathop\oplus\limits_{k=s}^\infty L^A_k$. We will also need the completion of this Lie algebra
\[
\aLA_+ = A \widehat{\otimes} L_+ = \varprojlim\limits_{s} L^A_{\geq s} = \fm \Der_{A}(A[[X_1,\dots,X_n]]) .
\]
In~\cite{BB24}, the authors constructed a topologically quasi-coherent sheaf $\hcL_+$ on $X$ of \emph{virtual jets of vector fields}. Locally, 
\[
\hcL_+(U) = \aL^{\cO(U)}_+ =\left\{ \sum_{i=1}^n \sum_{k \in \Z_+^n \setminus \{0\}} f_{k,i} X^k \xpartial{X_i} \mid  f_{k,i}\in\cO(U)\right \}
\]
for each étale chart $U \subset X$.  For each $s$, $\aLA_{> s} = \fm^{s+1} \aL_+^A$ is an ideal of the Lie algebra $\aL_+^A$. This ideal induces a subsheaf on $X$ with sections $\cO(U) \widehat{\otimes} \mathrm{L}_{> s}$ for each étale chart $U \subset X$. Furthermore, the quotient $\aLA_+/\aLA_{>s}$ induces a coherent sheaf $\cL^s$ of Lie algebras with sections $\cL^s(U) \cong \cO(U) \otimes L^s$ in étale charts, where $L^s = \aL_+/\aL_{>s}$.

Let $U$ be an étale chart of $X$ with uniformizing parameters $x_1,\dots,x_n$, let $\aD$ be the algebra of differential operators on $U$, $A= \cO(U)$, and $\cV=\Theta(U)$. Consider the kernel $\Delta$ of the multiplication map $A \otimes A \rightarrow A$ and define the Lie algebra of jets of vector fields as
\[
\widehat{J}_U = \varprojlim_{k}A \# \cV / ( \Delta^k \otimes_{A} \cV).
\]
The completion of $A \# U(\cV)$ is 
\[
\sAV(U) =  A \# U(\widehat{J}_U) / \left< f \# s - 1 \# fs \mid f \in A, \ s \in \widehat{J}_U \right> . 
\]
By~\cite{BI23}, these completions glue into a topologically quasi-coherent sheaf of associative algebras $\sAV$ on $X$.

\begin{theorem}[\cite{BI23}]\label{theorem:localisotheorem}
Let $U$ be an étale chart of $X$ with uniformizing parameters $x_1,\dots,x_n$. Let $\aD$ be the algebra of differential operators on $U$. Then
\[ 
\sAV(U) \cong \aD \otimes_A U_A(\aLA_+).
\]
The isomorphism map is given by:
\[
\varphi\left(g \# f \xpartial{x_i} \right) =g f \xpartial{x_i} \otimes 1
+ \sum_{k \in \Z_{+}^N \backslash \{ 0 \}} \frac{1}{k!} \, g \frac{\partial^k f}{\partial x^k} \otimes X^k \xpartial{X_i},
\]
and its inverse by
\begin{align*}
\psi\left(g \xpartial{x_i} \otimes 1 \right) &= g \# \xpartial{x_i}, \\
\psi\left(g \otimes X^m \xpartial{X_i} \right) &=  (g \# 1) (1 \# x - x \# 1)^m \xpartial{x_i}.
\end{align*}
\end{theorem}
In~\cite{BB24}, the authors gave the following explicit formula for the transformation law of the right-hand side of the isomorphism in Theorem~\ref{theorem:localisotheorem}. Suppose $U_1$ and $U_2 $ are étale charts with uniformizing parameters $x_1,\dots,x_n$ and $y_1,\dots,y_n$, respectively, and assume that 
\[
x_i = G_i(y), \quad y_i = H_i(x), \quad i =1,\dots,n,
\]
in $U_1 \cap U_2$. We will use the let symbols $X^k\xpartial{X_i}$ and $Y^k \xpartial{Y_i}$ to represent elements of $\aL_+$ in $\sAV(U_1)$ and $\sAV(U_2)$, respectively. Then, the transformation law for $\sAV$ from $U_1$ to $U_2$ is given by
\[
g(X)\xpartial{X_p} = \sum_{q=1}^n g( G(y+Y) - G(y) ) \fxpartial{H_q}{x_p} (G(y+Y)) \xpartial{Y_q},
\]
\[
\xpartial{x_i} = \sum_{j=1}^n \fxpartial{H_j}{x_i}(G(y)) \xpartial{y_i} + \left (\fxpartial{H_j}{x_i}(G(y+Y)) - \fxpartial{H_j}{x_i}(G(y)) \right ) \xpartial{Y_j},
\]
where 
\[
f(x+X) = \sum_{k \in \Z^n_+} \frac{1}{k!} \frac{\partial^k f}{\partial x^k}X^k
\]
denotes the Taylor expansion of $f$.

A sheaf $\cM$ on $X$ is called an $\sAV$-module if it is a quasi-coherent sheaf on $X$ that is a module over the sheaf of algebras $\sAV$. Similar to $\cD$-modules and $A\cV$-modules, the $\cO$-dual $\cM^{\ast} = \Hom_{\cO}(\cM,\cO)$ of an $\sAV$-module $\cM$ is an $\sAV$-module. Furthermore, the tensor product $\cM \otimes_{\cO} \cN$ is an $\sAV$-module if $\cM$ and $\cN$ are $\sAV$-modules.

The sheaf of algebras $\sAV$ contains a chain of sub-sheaves \[\mathcal{I}_1 \subset \mathcal{I}_2 \subset\cdots \subset \mathcal{I}_l \subset \cdots  \] with $\mathcal{I}_l(U)$ generated by $\Delta^l \otimes_A V$ in $\sAV(U)$. If $\cM$ is a sheaf of $A\cV$-modules, then we can always define an $\sAV$-module $\widehat{\cM}$ by 
\[
\widehat{\cM}= \varprojlim_{l} \cM / \mathcal{I}_l\cM.
\]
If $ \mathcal{I}_k\cM=0$ for large values of $k$, then $\widehat{\cM} = \cM$. Hence, the notions of $\sAV$-modules and sheaves of $A\cV$-modules coincide when this condition is met. This motivates the following definition.

\begin{definition}\label{definition:ndifferentiablemodule}
 Let $N\geq 1$. Let $X$ be affine and $M$ be an $A\cV$-module on $X$ with associated representation $\rho: \cV \rightarrow \End(M)$. We say that $M$ is $N$-\emph{differentiable} if
\[
 \sum_{k=0}^N (-1)^{N-k} \binom{N}{k} f^{k} \rho \left (f^{N-k}\eta \right ) =0
\]
 for every $f \in A$ and $\eta \in \cV$.

 Let $X$ be a quasi-projective variety. A sheaf of $A\cV$-modules $\cM$ is called $N$-differentiable if $\cM(U)$ is an $N$-differentiable $A\cV$-module on $U$ for each affine open set $U \subset X$. An $\sAV$-module $\cM$ is called $N$-differentiable if $\mathcal{I}_N \cM =0$.
\end{definition}

Note that an $A\cV$-module is $1$-differentiable if and only if it is a $\cD$-module. Furthermore, Definition~\ref{definition:ndifferentiablemodule} is equivalent to saying that $\rho$ is a differential operator of order less than or equal to $N-1$ in the sense of Grothendieck~\cite[Proposition 16.8.8]{Gro67}.

If $M$ is a differentiable $A\cV$-module on $X$, then $\widetilde{M}$ is an $\sAV$-module, as the next lemma shows.

\begin{lemma}{\cite[Proposition 26]{BI23}}\label{lemma:avmodulesheafifisdiff}
 Assume that $X$ is affine. Let $M$ be an $A\cV$-module with associated representation $\rho: \cV \rightarrow \End(M)$. If $M$ is $N$-differentiable, then the quasi-coherent sheaf $\widetilde{M}$ is both a sheaf of $A\cV$-modules and an $\sAV$-module. In particular, for each $f \in A$, the action of $A_f$ on $M_f=A_f \otimes_A M$ is given by localization, and the action of $\cV_f$ is given by
\[
\left( \frac{g}{f^k} \eta\right) m = \sum_{p=0}^{N} \sum_{l=0}^p (-1)^l \binom{p}{k} \frac{1}{f^{k(l+1)}} \rho\left(f^{kl} g \eta\right)m
\]
for every $g \in A$, $\eta \in V$, $k \geq 0$, and $m\in M_f$.
\end{lemma}

The next proposition shows that we may use uniformizing parameters to prove that a module is $N$-differentiable when $X$ is étale.
\begin{proposition}\label{lemma:ifxetalejustcheckunipar}
Let $\cM$ be an $\sAV$-module. Suppose that $U$ is an étale chart of $X$ with uniformizing parameters $x_1,\dots,x_n$. Let $A = \cO(U)$, $\cV = \Der(A)$, and $M = \cM(U)$ with associated representation $\rho: \cV \rightarrow \End(M)$. Fix $N \geq 1$. Then, the following statements are equivalent:
\begin{enumerate}
	\item  for each $f\in A$ and $\eta \in \cV$,
	\[
	\sum_{k=0}^N (-1)^{k}\binom{N}{k} f^{N-k} \# f^k\eta
	\]
	annihilates $M$;
    \item for each $i=1,\dots,n$, and $p \in \Z_+^n$ with $|p| \geq N$, 
    \[
 \sum_{k \in \Z_+^n, \ k \leq p} (-1)^{k}  \binom{p}{k} x^{p-k} \rho \left (x^{k}\xpartial{x_i} \right ) =0;
\]
    \item $\mathcal{I}_N(U)M=0$;
    \item $\aL_{\geq N-1}M =0$;
    \item $\rho: \, V \rightarrow \End \, M$ is a differential operator in the sense of Grothendieck of order less than or equal to $N-1$.
\end{enumerate}
\end{proposition}

\begin{proof}
By \cite[Theorem 22]{BI23}, the isomorphism between $\sAV(U)$ and $\cD(U) \otimes_A \, U_A(\aLA_+)$ sends
\begin{equation}
\label{Xp}
X^{p}\xpartial{X_i} \mapsto
 \sum_{k \in \Z_+^n, \ k \leq p} (-1)^{p-k}  \binom{p}{k} x^{p-k} \# x^{k}\xpartial{x_i}.
\end{equation}
It is straightforward to check that elements $X^{p}\xpartial{X_i}$ with $|p| \geq N$ 
and $i=1,\ldots,n$, generate ideal $\aLA_{\geq N-1}$ in Lie algebra $\aLA_+$.
Hence, $\aL_{N-1} M=0$ if and only if the action of the right-hand side on $M$ is zero for each $p\in \Z_+^n$ with $|p| \geq N$.
Thus, (2) and (4) are equivalent.

 Let us show that (4) implies (1). Suppose (4) holds. Then, $X^p \xpartial{X_i}M=0$ for each $|p| \geq N$. By Faà di Bruno's formula (see~\cite{FdB55}), we have that
\[
\frac{\partial^l (f^{k}g)}{\partial x^l} = \sum_{m=0}^{\min(|l|,k)} \frac{k!}{(k-m)!} f^{k-m} B(l,m,f,g),
\]
where $B(l,m,f,g)$ are polynomials in the partial derivatives of $f$ and $g$. Because
\[
\sum_{k=m}^{N}  (-1)^{k}\binom{N}{k} \frac{k!}{(k-m)!} = 0 
\]
if $m < N$, we can use the isomorphism of~\cite[Theorem 22]{BI23} and the Faà di Bruno's formula to get that  
\begin{align*}
 &\sum_{k=0}^N (-1)^{k}  \binom{N}{k} f^{N-k} \rho \left (f^{k} g \xpartial{x_i} \right ) \\
 =  &\sum_{k=0}^N (-1)^{k}  \binom{N}{k} \rho \left ( f^{N} g \xpartial{x_i} \otimes 1 \right ) 
 +  \sum_{k=0}^N \sum_{0 < |l| < N} (-1)^{k} \binom{N}{k} \frac{1}{l!} \rho\left( f^{N-k} \frac{\partial^l (f^{k}g)}{\partial x^l} \otimes X^l\xpartial{X_i} \right) \\
 = &  \sum_{0< |l| < N} \frac{1}{l!} \sum_{m=0}^{|l|}\left (  \sum_{k=m}^N  (-1)^{k}  \binom{N}{k} \frac{k!}{(k-m)!} f^{N-m}  \right ) \rho \left( B(l,m,f,g) \otimes X^l\xpartial{X_i} \right) = 0. 
\end{align*} 
Hence, (1) holds.

 Now assume that (1) holds for $N$. Then (1) also holds for all $N^\prime \geq N$: 
\[
 \sum_{k=0}^{N^\prime} (-1)^{k}  \binom{N^\prime}{k} f^{N^\prime-k} \rho \left (f^{k} g \xpartial{x_i} \right ) =0
\]
for every $f, g \in A$, $i=1,\dots,n$. In particular, this is true if $\displaystyle f= \sum_{i=1}^n a_i x_i$ and $g = 1$. By a linear algebra argument (see~\cite[Lemma 2.6.5]{Roc24}), for every $p \in \Z^n_+$ with $|p| \geq N$,
\[
 \sum_{k \leq p} (-1)^{k}  \binom{p}{k} x^{p-k} \rho \left (x^{k}\xpartial{x_i} \right ) =0.
\]
 Thus, (1), (2), and (4) are equivalent.

By \cite[Lemma 12]{BI23}, ideal $\Delta^N$ in $A \otimes A$ is generated as a left module
by elements of the form $(f_1 \otimes 1 - 1 \otimes f_1) \cdot \ldots \cdot (f_N \otimes 1 - 1 \otimes f_N)$, where $f_1, \ldots, f_N \in A$. Applying the linearization argument, we conclude that elements of the form $(f \otimes 1 - 1 \otimes f)^N$, 
$f \in A$, also generate $\Delta^N$ as a left module. Since $\mathcal{I}_N$ is generated by $\Delta^N \otimes_A V$, we conclude that (1) implies (3). Since elements (\ref{Xp}) with $|p| \geq N$ belong to $\mathcal{I}_N$, we obtain that (3) implies (2).

Finally, the equivalence of (1) and (5) follows from the definition of a differential operator in the sense of Grothendieck.
\end{proof}
\begin{remark}
Condition (4) in Proposition \ref{lemma:ifxetalejustcheckunipar} may be replaced with ($4^\prime$) $X^{p}\xpartial{X_i} M = 0$ for all $|p| = N$, $i=1,\ldots,n$, except for the case $n=1$ and $N=2$. Unless $n=1$ and $N=2$, such elements generate ideal $\aLA_{\geq N-1}$. In case when $n=1$ and $N=2$, we need to require $X^2 \xpartial{X} M = 0$ and $X^3 \xpartial{X} M = 0$.
\end{remark}

\begin{corollary}\label{corollary:equivndiffsheafandavhatmod}
There is an equality of categories between the category of $N$-differentiable sheaves of $A\cV$-modules and $N$-differentiable $\sAV$-modules.
\end{corollary}

An $\sAV$-module $\cM$ is $N$-differentiable if and only if $\cM$ is locally a module over $\aD \otimes U(\aL_+/\aL_{\geq N-1})$. By~\cite{BI23}, these algebras define a quasi-coherent sheaf of algebras, which we denote by $\mathcal{AV}^{N}$. We refer to $\cM$ as an $\sAVd$-module if it is differentiable. In view of Corollary \ref{corollary:equivndiffsheafandavhatmod}, and since every finite $A\cV$-module is differentiable, the category of $\sAVd$-modules naturally appears as the primary subcategory of $\sAV$-modules to be investigated.	We point out that the category of differentiable $\sAV$-modules is not semisimple, and there exist $N$-differentiable modules for every $N$.

The next example shows that there exist $A\cV$-modules that cannot be turned into a sheaf of $A\cV$-modules.
\begin{example}
Suppose that $X = \Spec(A)$ with $A=\bk[x,x^{-1}]$, and let $L(\lambda)$ be the simple highest weight module over $\cV = \Der(\bk[x,x^{-1}])$ with highest weight $\lambda \in \bk^{\times}$. Then, $M= A \otimes L(\lambda)$ is an $A\cV$-module on $X$, where the $A$-module action is given by left-side multiplication and
\[
\eta(f \otimes w) = \eta(f) \otimes w + f \otimes \eta w, \ \eta \in \cV, \ f \in A, \ w \in L(\lambda).
\]
This is the induced module $M = A \# U(\cV) \otimes_{U(\cV)} L(\lambda)$. We claim that the $A\cV$-module $M$ is simple but not differentiable. We leave this as an exercise for the reader. The $\cO$-module $\widetilde{M}$ is not a sheaf of $A\cV$-modules and the $\sAV$-module $\widehat{\cM}$ is equal to zero.

\end{example}

\begin{lemma}\label{lemma:misdiffifquoandsubarediff}
Let $\cM$ be an $\sAV$-module and $\mathcal{N} \subset \cM$. If $\mathcal{N}$ and $\cM/\mathcal{N}$ are differentiable, then $\cM$ is differentiable.
\end{lemma}
\begin{proof}
Suppose that $\mathcal{N}$ is $r$-differentiable, $\cM/\mathcal{N}$ is $s$-differentiable and $U\subset X$ an étale chart. Set $k=\max \{r,s \}$. Because
$[\hcL_{k-1} (U),\hcL_{k-1}(U)]=\hcL_{2k-2}$ if $n>1$ and \break
$[\hcL_{k-1} (U),\hcL_{k-1}(U)]=\hcL_{2k-1}$ if $n=1$, we have that $\hcL_{2k-1}(U)\cM(U)=0$. We conclude that $\cM$ is $2k$-differentiable by Proposition \ref{lemma:ifxetalejustcheckunipar}.
\end{proof}

We call an $\sAV$-module $\cM$ \emph{simple} if it is non-zero and its only $\sAV$-subsheaves are the trivial subsheaves $0$ and $\cM$, and $\cM(U)$ is either zero or a simple module over $\sAV(U)$ for each étale chart $U \subset X$.

If $\cM$ is an $\sAV$-module such that $\sAV(U)$ is a simple module over $\sAV(U)$ for each étale chart $U$ in some cover of $X$, then it is not necessarily true that $\cM$ is a simple $\sAV$-module, as the following example shows.

\begin{example}
Suppose that $X= \mathbb{P}^1$ with the usual cover $U_0 = \Spec(\bk[x])$ and $U_1 = \Spec(\bk[y])$ with the gluing $x \mapsto y^{-1}$. Consider the $\sAV$-module by gluing the $\aAV$-module on $U_0$ $\bk[\partial_x]\delta_0$ of $\delta$-functions supported at $0 \in U_0$ and the $\aAV$-module on $U_1$ $\bk[\partial_y]\delta_{\infty}$ of $\delta$-functions supported at $y=0$. Note that $\cM(U_0\cap U_1) = 0$, so the gluing is trivial. Although $\cM(U_0)$ and $\cM(U_1)$ are simple $A\cV$-modules on $U_0$ and $U_1$, respectively, $\cM$ is not a simple $\sAV$-module.
\end{example}

\begin{remark}
Suppose $\cM$ is an $\sAV$-module with no nontrivial $\sAV$-subsheaves. We conjecture that, for every étale chart $U \subset X$, the $\sAV(U)$-module $\cM(U)$ is either zero or simple. In this paper, we will work under the assumption that both conditions hold for a simple $\sAV$-module. The analogous statement is true for holonomic $\cD$-modules
\cite[Proposition 3.1.7]{HTT08}
\end{remark}

\section{Holonomic $\sAVd$-modules}\label{section:holavmodules}

Suppose that $U \subset X$ is an étale open set with uniformizing parameters $x_1,\dots,x_n$. Let $A$ be the algebra of regular functions on $U$. In an étale chart,  $\cV=\Der(A)$ is a free module over $A$ with basis $\partial_{x_1},\dots,\partial_{x_n}$. Let $\aD = \cD(U)$ be the algebra of differential operators on $U$. Then $\aD$ is a filtered algebra with $\aD_0 = A$ and 
\[
\aD_k = \left \{T \in \End_{\bk}(A) \mid [T,g] \in \aD_{k-1} \ \forall g \in A \right \}.
\]
Here $D_k$ is the space of differential operators of order $k$ in the sense of Grothendieck.
We have that 
\[
\rmgr \aD \cong \mathrm{Sym}_A \cV \cong A[\xi_{1},\dots,\xi_{n}],
\]
where $\xi_i$ is the image of $\partial_{x_i}\in \aD$ in $\rmgr \aD$. In particular, this filtration is a good filtration on $\aD$. The cotangent bundle is $\mathrm{Var}(\rmgr \aD)$ and it is isomorphic to $U \times \A^n$.

Let $M$ be a $\aD$-module. Suppose that $M$ has a good filtration. Then, 
\[
\Ann_{\rmgr \aD}\rmgr M= \left \{ x \in \rmgr \aD \mid xv = 0 \ \forall v \in \rmgr M \right \}
\]
is an ideal of $\rmgr \aD$. The \emph{characteristic variety} $\rmCh(M)$ is the algebraic variety
\[
\rmCh(M) = \mathrm{Var}\left ( \sqrt{\Ann_{\rmgr \aD}\rmgr M} \right ).
\]
The characteristic variety is independent of the choice of a good filtration on $M$. Because every finitely generated $\aD$-module admits a good filtration, the characteristic variety is well-defined for every finitely generated $\aD$-module. It is well-known that the dimension of irreducible components of $\rmCh(M)$ is greater than or equal to $n$, see~\cite[Corollary 2.3.2]{HTT08}. The dimension of $\rmCh(M)$ is equal to the Krull dimension of $\rmgr \aD /  \sqrt{\Ann_{\rmgr \aD}\rmgr W}$, which is equal to the Gelfand-Kirillov dimension of $M$.

\begin{remark}
If $M$ is finitely generated as an $A$-module, then $M_d =M$ for $d \geq 0$ is a good filtration on $M$. This means that 
\[
\Ann_{\rmgr \aD}\rmgr
 M \cong  (\xi_1,\dots,\xi_n) \subset A[\xi_1,\dots,\xi_n],
\]
thus $\rmCh(M)$ is contained in the zero section $X \times \{0\}$ of the cotangent bundle.
\end{remark}

Let $\cM$ be an $\sAV$-module. We define the \emph{GK-dimension} $\rmd(\cM)$ of $\cM$
\[
\rmd(\cM) = \sup_{p \in X} \GKdim_{(\sAV)_p} (\cM_p).
\]
Consequently, if $\{ U_{\alpha} \}$ is a finite étale cover of $X$, then 
$$\rmd(M) = \max_{\alpha} \GKdim_{\sAV(U_{\alpha})} (\cM(U_{\alpha})).$$

\begin{lemma}\label{lemma:bersteininequality}
If $\cM$ is a non-zero $\sAV$-module, then $ \rmd(\cM) \geq n$.
\end{lemma}
\begin{proof}
Without loss of generality, we may assume that $X$ is an étale chart. Using the notation above, $M = \cM(X)$ is a module over $\sAV(X)\cong \aD \otimes_A U_A(\aLA_+)$. If $v \in M$ is non-zero, then $P = \aD v$ is a finitely-generated $\aD$-module. Hence, by Lemma~\ref{lemma:commalgebrarelation},
\begin{align*}
	& \GKdim_{\sAV}(M)  \geq \GKdim_{\aD}(P) =  \GKdim_{\rmgr \aD} (\rmgr P) = \GKdim_{\rmgr \aD /  \sqrt{\Ann_{\rmgr \aD}\rmgr P}} (\rmgr P) \\
	= & \GKdim \left (\rmgr \aD /  \sqrt{\Ann_{\rmgr \aD}\rmgr P}  \right )  = \dim\rmCh(P)  \geq n.
\end{align*}
\end{proof}

\begin{definition}
We say that an $\sAV$-module $\cM$ on $X$ is holonomic if $\cM$ is locally finitely generated and $\rmd(\cM) = \dim X$ or $\cM=0$. 
\end{definition}

In this paper, we will study differentiable holonomic $\sAV$-modules. In light of the results discussed in Section \ref{section:sheafAVdiffmodules}, differentiable modules provide a perfect setting to develop this theory. 

Suppose $X$ is étale and $\aD$ is the algebra of differential operators on $X$. Let $P$ be a $\aD$-module and $W$ an $\aL_+$-module such that $\aL_{\geq N-1} W =0$ for some $N\geq 1$. By Theorem~\ref{theorem:localisotheorem}, \[T(P,W) = P \otimes W \] is an $\sAV(X)$-module. The action of $\Theta(X)$ on $T(P,W)$ is given by
\[
f \xpartial{x_i} (p \otimes w) = \left ( f\xpartial{x_i} p \right )\otimes w + \sum_{k\in \Z_+^n \setminus \{0\} } \frac{1}{k!} \frac{\partial^k f}{\partial x^k}p \otimes X^k\xpartial{X_i} w,
\]
and $\cO(X)$ acts on $T(P,W)$ in the left factor. Note that the sum in the above formula is finite and $T(P,W)$ is an $N$-differentiable $A\cV$-module on $X$.
When $P = A$, we denote $T(A, W)$ simply by $T(W)$. Note that $T(P, W) \cong P\otimes_A T(W)$ as $\sAV(X)$-modules.

\begin{definition}
Let $W$ be an $\aL_+$-module such that $\aL_{\geq N-1}W =0$ for some $N \geq 1$. An $\sAV$-module $\cM$ is called \emph{of type $W$} if for each étale chart $U\subset X$ there exists a $\cD(U)$-module $P_U$ such that 
\[
\cM(U) \cong T(P_U,W).
\]
\end{definition}

The tensor module $\mathcal{J}^W$ constructed in \cite{BR24} (see Section~\ref{section:coherentavmod}) is an $\sAV$-module of type $W$. Furthermore, by~\cite{BB24}, a Rudakov module associated with a $\fm_p \Theta_p / \fm_p^N \Theta_p$-module $W$ is an $\sAV$-module of type $W \otimes \bk_1$, where $p\in X$ and $\fm_p$ denotes the maximal ideal of $\cO_p$ (see Example \ref{klambda} below for the definition of $\bk_1$).

Note that if $\cM$ is an $N$-differentiable $\sAV$-module of type $W$, then $W$ is a module over $L^{N-2} = \aL_+/\aL_{\geq N-1}$. In particular, if $\cM$ is $2$-differentiable then $W$ is a $\gl_n$-module.

\begin{lemma}\label{lemma:tensorholonomicgkdimlpluszero}
Suppose $X$ is étale and $\aD$ is the algebra of differential operators on $X$. Let $W$ be a finitely generated $\aL_+/ \aL_{\geq N-1}$-module and let $P$ be a holonomic $\aD$-module. If $T(P,W)$ is a holonomic $\sAV(X) \cong  \cD(U) \otimes_A \, U_A(\aLA_+)$-module, then $\GKdim_{U(\aL_+)}(W) = 0$ and $W$ is finite-dimensional. 
\end{lemma}

\begin{proof}
    $M = T(P,W)$ is a finitely generated module over the algebra $\aD \otimes U(L^{N-2})$. In particular, $M$ is a module over $\aD \otimes U(\mathrm{L}_+)$, where $\mathrm{L}_+ = (X_1,\dots,X_n)\Der(\bk[X_1,\dots,X_n])$. By \cite[Lemma 2.6]{BR24}, $\GKdim_{U(L_+)}(W) = 0$ and $\dim W < \infty$.
\end{proof}

\begin{lemma}\label{lemma:tpwholimplieswfindim} 
Suppose that $X$ is étale and let $\aD$ be the algebra of differential operators on $X$. Let $P$ be a $\aD$-module and $W$ be a finite-dimensional $\aL_+/ \aL_{\geq N-1}$-module. Then \[\GKdim_{\sAV(X)} T(P,W) = \GKdim_{\cD}(P). \]
\end{lemma}
\begin{proof}
Suppose that $Q \subset \sAV(X)$ is a subframe and $E \subset T(P, W)$ is a finite-dimensional subspace. Since $\sAV \cong D \otimes_A \, U_A(\aLA_+)$, there exist subframes $H \subset D$ and $K \subset U_A(\aLA_+)$, such that $Q \subset H \otimes K$. Likewise, there exists a finite-dimensional subspace $F \subset P$, such that $E \subset F \otimes W$.

Consider the action of $K$ on $T(W)$. There exists a subframe $G \subset A$, such that
$K (1 \otimes W \subset G \otimes W$. Then 
$Q E \subset HF \otimes_A (G \otimes W) = GH F \otimes W$.

Set $S = GH$, a subframe in $D$. By induction, we get $Q^k E \subset S^k F \otimes W$.
This implies that $\GKdim_{\sAV(X)} T(P,W) \leq \GKdim_{\cD}(P)$. Since for étale $X$, $D$ is a subalgebra in $\sAV(X)$, the opposite inequality holds as well.
\end{proof}

\begin{lemma}[Jacobson Density Theorem]
Suppose $\bk$ is uncountable. Let $P$ be a simple $\aD$-module, $\{p_1,\dots,p_s \} \subset P$ a linearly independent set and $w_1,\dots,w_s \in P$. Then, there exists $a \in \aD$ such that $a p_i = w_i$ for all $i = 1, \ldots, s$.
\end{lemma}
\begin{proof}
    $P$ has a countable dimension as a $\bk$-vector space because $\aD$ has a countable dimension and $P$ is a simple $\aD$-module. Since $\bk$ is uncountable, it follows from Schur's Lemma that $\End_{\aD}(P) \cong \bk$ (see~\cite[Section VIII.3.2 Theorem 1]{Bou23}). The lemma follows from this isomorphism and the Jacobson Density Theorem, see~\cite[Theorem 13.14]{Isa09}.
\end{proof}

\begin{proposition}\label{proposition:containstensorproduct}
Suppose that $X$ is an étale chart and $\aD$ is the algebra of differential operators on $X$. Let $M$ be a holonomic $N$-differentiable $\sAV(X)$-module. Then $M$ contains a simple $\sAV(X)$-submodule isomorphic to $T(P,W)$, where $P$ is a simple $\aD$-module and $W$ is a simple finite-dimensional $L^{N-2}$-module.
\end{proposition}
\begin{proof}
Let $M_0$ be a generating subspace of $M$. Since $M$ is a module over $D \otimes U(\aL_+/\aL_{\geq N-1})$, $\aD M_0$ is a holonomic $\aD$-module. Therefore, $\aD M_0$ contains a simple $\aD$-module $P$ because it has finite length by~\cite[Proposition 3.1.2]{HTT08}. Consider the homomorphism of $\sAV(X)$-modules
\begin{align*}
\Psi: T(P,U(L^{N-2})) & \rightarrow M\\
v \otimes u \mapsto uv.
\end{align*}

  Take an element $w \in \ker \Psi$ and write $w = \sum_{i=1}^k v_i \otimes u_i$ where $\{v_1,\dots,v_k\}$ is a linearly independent set. By the Jacobson Density theorem, there exists $p \in \aD  \subset \sAV(X)$ such that $pw = v_1 \otimes u_1$. Therefore, for every $w \in \ker \Psi$, 
  $\sAV(X) w= P \otimes I_w$ for some left ideal $I_w$ of $U(L^{N-2})$. We conclude that the image of $\Psi$ is a non-zero holonomic $A\cV$-submodule of $M$ isomorphic to $T(P,W')$, where 
\[
W' = U(L^{N-2})/ \sum_{w \in \ker \Psi} I_w.
\]
By Lemma~\ref{lemma:tensorholonomicgkdimlpluszero}, $\displaystyle \dim W'  < \infty$. Take a simple $\aL_+$-submodule $W\subset W'$. Thus $M$ has a submodule isomorphic to $T(P,W)$.
\end{proof}

\begin{proposition}
If $\cM$ is a holonomic differentiable $\sAV$-module, then $\cM$ has finite length.
\end{proposition}

\begin{proof}
	It is enough to prove that $M=\cM(U)$ has a finite composition series for each étale chart $U \subset X$. By hypothesis, $\cM$ is a differentiable $\sAV$-module. Hence, by Theorem \ref{theorem:localisotheorem} and Proposition \ref{lemma:ifxetalejustcheckunipar}, $M$ is a module over the algebra $R = \aD \otimes U(L^N)$ for some $N\geq 0$. Suppose $M_0\subset M$ is a finite-dimensional generating subspace of the $R$-module $M$, then $\aD M_0$ is a holonomic $\aD$-module. By \cite[Proposition 3.1.2]{HTT08}, $\aD M_0$ has finite length $l$. We will prove by induction on $l$ that $M$ has finite length as an $R$-module. The basis of induction is $l=0$, in which case there is nothing to prove. Suppose $l>0$. Then, $DM_0$ contains a simple $\aD$-submodule $S$. Let $v \in S$ be a generator of the $\aD$-module $S$ and let $W = U(L^N) v$. Then, using the same argument as in the previous proposition, we get $U(\aL_+)S \cong T(S, W)$ since $S$ is a simple $\aD$-module and the Jacobson Density Theorem holds for it. By Lemma \ref{lemma:tpwholimplieswfindim}, $W$ has finite dimension, hence $M' = U(L^N)S$ has a finite composition series. The $R$-module $M /M'$ is holonomic and generated by the image of $M_0$. Because the image of $\aD$-module $DM_0$ in $M/M'$ has length strictly less than $l$, we conclude by the induction hypothesis that both $M/M'$ and $M'$ have finite composition series. Therefore, $M$ has a finite composition series as well.
\end{proof}

\begin{corollary}
Suppose that $X$ is étale and $M$ is a differentiable holonomic $A\cV$-module on $X$. If
\[
0 = M_0 \subset M_1 \subset \cdots \subset M_l = M
\]
is a composition series of $M$, then there exist simple finite-dimensional $\gl_n$-modules \break 
$W_1,\dots,W_l$ and simple holonomic $\cD(X)$-modules $P_1,\dots,P_l$ such that
\[
M_i / M_{i-1} \cong T(P_i,W_i)
\]
for each $i=1,\dots,l$.
\end{corollary}

\begin{conjecture}
Every holonomic $\sAV$-module is differentiable.
\end{conjecture}

\section{$\sAV$-modules that are coherent $\cO$-modules}\label{section:coherentavmod}

In this section, we study $\sAV$-modules that are coherent as $\cO$-modules. We will also further investigate the sheaves of tensor modules introduced in~\cite{BR24}.
\begin{proposition}
An $\sAV$-module that is coherent over $\cO$ is holonomic.
\end{proposition}

\begin{proof}
	Let $U \subset X$ be an étale chart with local parameters $x_1,\dots,x_n$. Then, $\cM(U)$ is an $A\cV$-module on $U$. By~\cite{BR23}, $\cM(U)$ is an $N$-differentiable $A\cV$-module on $U$, where $N$ depends on the rank of $\cM(U)$ as an $\cO(U)$-module. By Proposition \ref{lemma:ifxetalejustcheckunipar}, $\cM(U)$ is a module over $\cD(U) \otimes U(L^{N-2})$. Suppose that $F_A$ is a frame of $\cO(U)$ and $F_B$ is the subspace spanned by the partial derivatives with respect $x_1,\dots,x_n$ and 
	\[
	X^k \xpartial{X_i}, \quad i=1,\dots, n, \ k\in\Z_+^n \text{ with } |k|  < N.
	\]
	Then, $F = F_A + F_B$ is a frame for $\cD(U) \otimes U(L^{N-2})$. If $M_0$ is a finite-dimensional generating subspace of $M$ as an $\cO(U)$-module, then there exists $s\geq 1$ such that $F_BM_0 \subset F_A^sM_0$. Hence, for each $k \geq 0$,
	$
	F^k M_0 \subset F^{ks}_AM_0.
	$
	Thus,
	\[
	n \leq \GKdim_{\sAV(U)} \cM(U) = \overline{\lim_k} \log_k \dim F^kM_0 \leq  \overline{\lim_k} \log_k \dim F_A^{ks}M_0 \leq \GKdim\, \cO(U)=n.
	\]
\end{proof}
\begin{corollary}
An $\sAV$-module that is coherent over $\cO$ is a locally free $\cO$-module, i.e., a vector bundle on $X$.
\end{corollary}
\begin{proof}
Let $\cM$ be an $\sAV$-module that is coherent over $\cO$. By the above proposition, $\cM$ is differentiable and a sheaf of $A\cV$-modules. Suppose that $U \subset X$ is an affine open set such that $\cM|_U = \widetilde{M}$ for some $A\cV$-module $M$ on $U$ that is finitely generated as an $\cO(U)$-module. By~\cite[Proposition 1.2]{BR23}, $\cM(U)$ is a projective $\cO(U)$-module. In particular, $\cM_{p}$ is a free $\cO_{X,p}$-module for each $p \in U$. We conclude that $\cM$ is a locally free $\cO$-module, i.e., $\cM$ is a vector bundle \cite{Ser55}.
\end{proof}

Therefore, similar to what happens in $\cD$-modules, an $\sAV$-module is coherent over $\cO$ if and only if it is locally free over $\cO$ of finite rank. In particular, if $X$ is affine, then global sections of an $\sAV$-module that is coherent over $\cO$ form a projective $\cO(X)$-module.

In~\cite{BR24}, the authors constructed a sheaf $\mathcal{J}^W$ such that $\mathcal{J}^W(U) \cong T(\cO(U),W)$, where $W$ is a finite-dimensional representation of the Lie algebra $\aL_+$ for which the action of $\gl_n$ integrates to a rational $\GL_n$-module. The $\sAV$-module $\mathcal{J}^W$ is called a \emph{tensor module}. We will recall this construction, assuming that $W$ is a representation of the Lie algebra $L^0 \cong \gl_n$.

Denote by $\h \subset \gl_n$ the standard Cartan subalgebra of $\gl_n$ of diagonal matrices. Let $\rho:\gl_n \rightarrow \gl(W)$ be a simple finite-dimensional $\gl_n$-module with highest weight $\mu \in \h^{\ast}$ such that $\mu(E_{ii}) \in \Z$ and $\mu(E_{ii}) \geq \mu(E_{i+1,i+1})$, i.e., $\mu$ is integral dominant. Then $\rho$ integrates to a rational representation $\widetilde{\rho}: \GL_n \rightarrow \GL(W)$ of the Lie group $\GL_n$, and $W$ is called \emph{integrable}.

Let us construct the tensor module $\mathcal{J}^W$ associated with $W$. For each étale chart $U \subset X$, define $\mathcal{J}^W(U) = \cO(U) \otimes W$ as a vector space. The action of $\cO(U)$ on $\mathcal{J}^W(U)$ is given by left multiplication and the action of $\Theta(U)$ is given by
\[
\left ( f \xpartial{x_i}\right )(g \otimes w) = \fxpartial{g}{x_i}\otimes w + g \sum_{j=1}^n \fxpartial{f}{x_j} \otimes E_{ji}w,
\]
where $x_1,\dots,x_n$ are the uniformizing parameters of $U$ and $f,g \in \cO(U)$. If $U_1$ and $U_2$ are étale charts with uniformizing parameters $x_1,\dots,x_n$ and $y_1,\dots,y_n$, then the Jacobian matrix $\varphi_{U_1U_2}$ with entries 
$\, \partial x_i / \partial y_j \ i,j=1,\dots,n,$
is an element of $\GL_n(\cO(U_1 \cap U_2))$. Hence, $\widetilde{\rho}(\varphi_{U_1U_2})$ is a well-defined $\cO(U_1\cap U_2)$-linear automorphism of $\cO(U_1 \cap U_2)\otimes W$ by~\cite[Lemma 5]{BR24}. To glue $\cJ^W(U_1)$ and $\cJ^W(U_2)$, we use the map
\[
 gw \mapsto g \widetilde{\rho}(\varphi_{U_1U_2})w, \quad g \in \cO(U_1 \cap U_2), \ w \in W.
\]
By~\cite[Theorem 17]{BR24}, $\cJ^W$ is a $2$-differentiable $\sAV$-module and $\mathcal{J}^W(U) \cong T(\cO(U),W)$ for every étale chart $U \subset X$. Furthermore, $\cJ^W$ is coherent over $\cO$ and a simple $\sAV$-module because $W$ is a simple finite-dimensional $\gl_n$-module, see \cite{BNZ21}.

\begin{example} \label{klambda}
Let $\{U_{\alpha} \}$ be a finite étale cover of $X$ such that $x^{\alpha}_{1},\dots,x^{\alpha}_{n}$ are uniformizing parameters for $U_{\alpha}$. Suppose that $\lambda \in \Z$. Then, the one-dimensional $\gl_n$-module $\bk_{\lambda}=\bk v_{\lambda}$ given by $T v_{\lambda} = \lambda \mathrm{tr}(T)  v_{\lambda}$,  is integrable. Denote by $v_{\lambda}^{\alpha} = 1 \otimes v_{\lambda} \in J^{\bk_{\lambda}}(U_{\alpha})$. We have
\[
\left ( f \xpartial{x^{\alpha}_i} \right) (g v_{\lambda}^{\alpha}) = \left ( f\fxpartial{g}{x^{\alpha}_i} + \lambda g \fxpartial{f}{x^{\alpha}_i}\right  )v_{\lambda}^{\alpha}
\]
for each $f,g \in \cO(U_{\alpha})$ and $i=1,\dots,n$. If $J_{\alpha\beta}$ denotes the determinant of the matrix with coefficients
$
\, \partial x^{\alpha}_i / \partial {x^{\beta}_j}, \ i,j=1,\dots,n,
$
then the transformation law for $\mathcal{J}^{\bk_{\lambda}}$ from $U_{\alpha}$ to $U_{\beta}$ is the map
$
v_{\lambda}^{\alpha} \mapsto J_{\alpha \beta}^{\lambda} v_{\lambda}^{\beta}$.
\end{example}

Any $\fsl_n$-module can be made into a $\gl_n$-module by setting the action of the identity matrix to zero. On the other hand, if $W$ is a simple finite-dimensional $\gl_n$-module, then the identity matrix acts as $n\lambda$ for some $\lambda \in \bk$. Hence, if $W$ is a simple finite-dimensional $\gl_n$, then there exists a simple finite-dimensional $\fsl_n$-module $W_0$, viewed as a $\gl_n$-module with a trivial action of the identity matrix, and $\lambda \in \bk$ such that
\[
W \cong W_0 \otimes \bk_{\lambda}.
\]
The $\gl_n$-module $W$ will be integrable if and only if $\lambda$ is an integer. In particular, the tensor module $\mathcal{J}^{W_0}$ is well-defined for any finite-dimensional $\fsl_n$-module $W_0$.

\section{Simple holonomic modules}\label{section:irreducibleholmodules}

In this section, we further investigate differentiable holonomic $\sAV$-modules. Our objective is to have a better understanding of the structure of simple $\sAV$-modules. Recall that an $\sAV$-module $\cM$ is simple if it is non-zero and its only $\sAV$-subsheaves are the trivial subsheaves $0$ and $\cM$, and $\cM(U)$ is either zero or a simple module over $\sAV(U)$ for each étale chart $U \subset X$. The main objective of this section is to show that we can associate to every simple holonomic $\sAV$-module $\cM$, a simple $\gl_n$-module $W$. We will then use this structure to decompose $\cM$ into the tensor product of two sheaves.

We begin by providing a local description of simple holonomic $\sAV$-modules.

\begin{proposition}\label{proposition:sheafofavmodulearetensormod}
Let $\cM$ be a simple differentiable holonomic $\sAV$-module. Then for a finite open cover $\{ U_{\alpha}\}$ of $X$ consisting of étale charts, $\cM$ is either zero or 
\[
\cM(U_{\alpha}) \cong T(P_{\alpha},W_{\alpha}),
\]
where $P_{\alpha}$ is a simple holonomic $\cD(U_{\alpha})$-module and $W_{\alpha}$ is a simple finite-dimensional $\gl_n$-module.
\end{proposition}
\begin{proof}
Since $X$ is a smooth quasi-projective variety, it admits a finite open cover $X=\bigcup U_{\alpha}$ of étale charts. For each $\alpha$, $\cM(U_{\alpha})$ is a simple $\sAV(U_{\alpha})$-module, or $\cM(U_{\alpha}) = 0$. If $\cM(U_{\alpha}) \neq 0$, by Proposition~\ref{proposition:containstensorproduct}, $\cM(U_{\alpha}) \cong T(P_{\alpha},W_{\alpha})$, where $P_{\alpha}$ is a simple holonomic $\cD(U)$-module and $W_{\alpha}$ is a simple finite-dimensional $\gl_n$-module. 
\end{proof}

Locally, in an étale chart, a simple holonomic $\sAV$-module is a tensor product of a $\cD$-module with a simple $\gl_n$-module. However, the $\cD$-module factor depends on the choice of uniformizing parameters, even when $X$ is étale. This reflects the notion of $\cL_+$-charged $\cD$-module in \cite{BB24}. Example~\ref{example:globalsectionwithlocalpar} below exhibits this behavior explicitly.   

Proposition~\ref{proposition:sheafofavmodulearetensormod} gives a local description of a simple differentiable holonomic $\sAV$-module. However, it is not obvious that the $\gl_n$-modules $W_{\alpha}$ that appear in Proposition~\ref{proposition:sheafofavmodulearetensormod} are isomorphic to each other. Next, we aim to prove this claim. For this, we analyze the action of the sheaf $\widehat{\cL}_+$, which will be reduced to an action of a $\gl_n$-bundle $\Theta \otimes_{\cO} \Omega^1$, focusing, in particular, on the action of the center of the algebra $U(\gl_n)$. We first recall key results concerning the center of $U(\gl_n)$ and its action on $\gl_n$-modules.

Let $E_{ij}$, $i,j=1,\dots,n$, be the standard basis of $\gln$. For each $k = 1,\dots,n$, we denote
\[
\Omega_k  = \sum_{i_1,\dots,i_k = 1,\dots,n} E_{i_1 i_2}E_{i_2 i_3}\cdots E_{i_k i_1}.
\]
By~\cite[Corollary 7.1.2]{Mol07}, the center of $U(\gl_n)$ is $\bk[\Omega_1,\dots,\Omega_n]$.

\begin{lemma}\label{lemma:twoglnisisoifcharequal}
Let $W_1$ and $W_2$ be two simple finite-dimensional $\gl_n$-modules. Then there exist scalars $\lambda_1,\dots,\lambda_n,\mu_1,\dots,\mu_n \in \bk$ such that $\Omega_k$ acts on $W_1$ as multiplication by $\lambda_k$ and on $W_2$ as multiplication by $\mu_k$, for every $k=1,\dots,n$. Furthermore, if $\lambda_k = \mu_k$ for all $k$, then $W_1$ and $W_2$ are isomorphic.
\end{lemma}
\begin{proof}
By Schur's Lemma, central elements act on simple modules by scalars.
Furthermore, by~\cite{Hum08}, a simple finite-dimensional $\gl_n$-module is determined by the central character $\chi: Z(U(\gl_n)) \rightarrow \bk$ given by the action of the center on it.
\end{proof}

The $\gl_n$-bundle $\Theta \otimes_{\cO} \Omega^1$ of $(1,1)$-tensors induces a sheaf $\glnbund$ of associative algebras on $X$, whose sections in étale charts $U_\alpha \subset X$ are isomorphic to $\cO(U_\alpha) \otimes U(\gl_n)$. 

\begin{theorem}\label{theorem:centerisasubsheaf}
The center of the sheaf $\glnbund$ is a trivial sheaf $\Zglnbund$ with 
$\Zglnbund (U_\alpha) = \cO(U_\alpha) \otimes \bk[\Omega_1,\dots,\Omega_n]$ for each 
étale chart $U_\alpha$ and trivial transformation laws between the charts.
\end{theorem}
\begin{proof}
Suppose $U_\alpha$ and $U_\beta$ are elements of the étale cover with uniformizing parameters $x_1,\dots,x_n$ and $y_1,\dots,y_n$, respectively. Explicitly, the transformation law $\psi_{\alpha\beta}$ for the sheaf $\glnbund$ is
\[
f \otimes E_{ij} \mapsto  \sum_{a=1}^n \sum_{b=1}^n f \fxpartial{y_a}{x_j} \fxpartial{x_i}{y_b} \otimes E_{ba}.
\]
Since 
\[
\delta_{ij} = \fxpartial{y_i(x(y))}{y_j}= \sum_{l=1}^n \fxpartial{y_i}{x_l} \fxpartial{x_l}{y_j},
\]
it follows that
\begin{align*}
& \psi_{\alpha\beta} \left(\sum 1 \otimes E_{i_1i_2}E_{i_2i_3}\cdots E_{i_ki_1} \right ) \\
= & \sum \fxpartial{y_{a_k}}{x_{i_1}} \fxpartial{x_{i_1}}{y_{b_1}} \fxpartial{y_{a_1}}{x_{i_2}} \fxpartial{x_{i_2}}{y_{b_2}} \cdots \fxpartial{y_{a_{k-1}}}{x_{i_k}} \fxpartial{x_{i_k}}{x_{b_k}} \otimes E_{b_1a_1}E_{b_2a_2}\cdots E_{b_k a_k} \\
= & \sum \delta_{b_1a_k}\delta_{a_1b_2}\delta_{a_2b_3}\cdots \delta_{b_ka_{k-1}} \otimes E_{b_1a_1}E_{b_2a_2}\cdots E_{b_k a_k} \\
 = & \sum 1 \otimes E_{b_1b_2}E_{b_2b_3}\cdots E_{b_kb_1}.
\end{align*}
We conclude that $\psi_{\alpha\beta}(1 \otimes \Omega_k) = 1 \otimes \Omega_k$. Therefore, the center of $\glnbund$ is a trivial sheaf $\Zglnbund= \cO \otimes Z(U(\gl_n))$.
\end{proof}
	Let $Z_{\chi}$ be the kernel of the character $\chi: Z(U(\gl_n)) \rightarrow \bk$. With a slight abuse of notations, we identify $Z_{\chi}$ with the corresponding subspace of constant sections $1 \otimes Z_{\chi}$ in the sheaf $\mathcal{Z}$, constructed above. Note that for any $\sAV$-module $\cM$, the kernel of $Z_{\chi}$ in $\cM$ is an $\sAV$-submodule in $\cM$.

We are ready to state our second result about the local description of irreducible holonomic $\sAV$-modules. 

\begin{proposition}\label{proposition:sameglnmoduleineverychart}
    Let $\cM$ be a simple holonomic $\sAV$-module. Then there exists a simple $\gl_n$-module $W$ such that $\cM(U) \cong T(P,W)$ for every étale chart $U\subset X$ for some simple $\cD(U)$-module $P$ or $P=0$. In particular, $\cM$ has type $W$.
\end{proposition}

\begin{proof}
	Let $U_1 \subset X$ be an étale chart such that $\cM(U)\cong T(P_1,W_1)$ is non-zero. Let $\chi$ be the character of $Z(U(\gl_n))$ corresponding to $W_1$. The kernel of $Z_{\chi}$ is a non-zero subsheaf in $\cM$ since $Z_{\chi}$ annihilates $T(P_1,W_1)$. Since $\cM$ is simple, $\cM = \ker Z_{\chi}$. Consider an étale chart $U_2$ with $\cM(U_2) \cong T(P_2,W_2)$. Since $Z_{\chi}$ annihilates $T(P_2,W_2)$ as well, we get that $W_1\cong W_2$.
\end{proof}

Let us recall a construction of a gauge module. Let $\cP$ be a $\cD$-module, which is a vector bundle, and let $W$ be a simple finite-dimensional $\gl_n$-module.
Let $\cM$ be a simple holonomic $\sAV$-module such that in étale charts $U \subset X$ we have $\cM (U) = \cP(U) \otimes W$ with $\cP(U)$ being a free $\cO(U)$-module. Suppose $\cP(U) = \cO(U) \otimes P^\prime$ for some finite-dimensional subspace $P^\prime \subset \cP(U)$. Then locally, the action of vector fields is given by
\[
f \xpartial{x_i} ( g p \otimes w) = f\fxpartial{g}{x_i}p \otimes w + fg \left (\xpartial{x_i} p \right ) \otimes w + g \sum_{j=1}^n \fxpartial{f}{x_j} p \otimes E_{ji}w
\]
for $f, g \in \cO(U)$, $p \in P'$, $w \in W$, $i=1,\dots,n$. An ingredient to this construction is the maps
\[
\xpartial{x_i} : P' \otimes W  \rightarrow \cO(U)\otimes P'\otimes W , \quad i=1,\dots,n,
\] 
which are called the \emph{gauge fields}.

\begin{example}\label{example:globalsectionwithlocalpar}
	Consider the affine variety $X=\Spec \bk[t,s]/(st-1)$. Define the $A\cV$-module $M_1 = \bk[t,t^{-1}]v_1$ on $X$ by
	\[
	t^k \xpartial{t} (t^l d) = (l+\lambda k)t^{k+l-1}v_1
	\]
	and the $A\cV$-module $M_2=\bk[s,s^{-1}]v_2$ by 
	\[
	s^k \xpartial{s} (s^l d_2) = \left (l+\lambda k \right )s^{k+l-1}v_2
	\]
	Both $M_1$ and $M_2$ are gauge modules of type $\bk_{\lambda}$. If we rewrite the action of vector fields on $M_1$ choosing $s$ as local parameter, we will have
	\[
	\left (s^k \xpartial{s} \right )s^lv_1 = \left (-t^{2-k}\xpartial{t} \right )t^{-l} v_1 = -(-l + (2-k)\lambda )t^{-l-k+1}v_1 = (l + (k-2)\lambda )s^{l+k-1}v_1
	\]
	for every $l,k\in \Z$. Hence, the gauge field $-2\lambda s^{-1}$ appears if the uniformizing parameter $s$ is chosen instead. If $2\lambda \in \Z$, then this gauge field may be trivialized by choosing $v_1' = s^{2\lambda} v_1$. In particular, $M_1$ and $M_2$ are isomorphic if and only if $\lambda$ is a half-integer, and the isomorphism is $v_2 \mapsto s^{2\lambda} v_1$. If $\lambda$ is not a half-integer, this change of uniformizing parameters yields gauge fields that cannot be trivialized. 
	
	Alternatively, we can realize $X \cong \Spec \bk[x,y]/(x^2+y^2-1)$ and choose $x \in \bk[x,y]/(x^2+y^2-1)$ as a uniformizing parameter in $D(y)=\{\mathfrak{p} \in X \mid y \notin \mathfrak{p} \}$. Again, if we consider $\widetilde{M_1}(D(y))$ with uniformizing parameter $x$, a gauge field will emerge, and it can only be trivialized if $\lambda \in \Z$. 
\end{example}
We can now move on to the main result of this section, where we prove that the modules $P_{\alpha}$ that appear in Proposition~\ref{proposition:sameglnmoduleineverychart} can be glued into a $\cD$-module if $W$ is integrable. 

\begin{theorem}\label{theorem:sheavesofavmodulesaretensorproducct}
	Let $\cM$ be a simple differentiable holonomic $\sAV$-module of type $W$. If $W$ is integrable, then there exists a simple $\cD$-module $\mathcal{P}$ such that 
\[ \cM \cong \mathcal{P} \otimes_{\cO} \mathcal{J}^W. \]
\end{theorem}

\begin{proof}
	Because $W$ is integrable, $\mathcal{J}^{W}$ and $\mathcal{J}^{W^\ast}$ are well-defined tensor modules. Let $Z_0$ be the space of constant sections in $\mathcal{Z}$ corresponding to the zero central character. Let $\mathcal{P}$ be the kernel of the action of $Z_0$ on $\cM \otimes_{\cO} \mathcal{J}^{W^{\ast}}$. We claim that $\mathcal{P}$ is a $\cD$-module and $\cM \cong \mathcal{P} \otimes_{\cO} \mathcal{J}^W$. Since $\mathcal{P}$ has type $\bk_0$, it is a $\cD$-module. Moreover, if $\cM(U) \cong T(P_U,W)$, then $\mathcal{P}(U) \cong P_U$. Therefore, $\cM \cong \mathcal{P} \otimes_{\cO}\mathcal{J}^{W}$ since these local isomorphisms are compatible with gluing by construction.
\end{proof}

The notion of an $\cL_+$-charged $\cD$-module was introduced in \cite{BB24}. In the case where the action of $\cL_+$ reduces to a $1$-dimensional representation $\bk_{\lambda}$, this definition simplifies to the following.

\begin{definition}
	Let $\lambda \in \bk$. We call a sheaf $\mathcal{F}$ a \emph{$\lambda$-charged $\cD$-module} if for each étale chart $U_1$ with uniformizing parameters $x_1,\dots,x_n$ we have a $\cD(U_1)$-module structure on $\mathcal{F}(U_1)$ and, on the intersection with another étale chart $U_2$ with uniformizing parameters $y_1,\dots,y_n$, the transformation law of differential operators is given by
	\[
	f\xpartial{x_i} \mapsto \sum_{j=1}^n f\fxpartial{y_j}{x_i} \xpartial{y_j} + \lambda f \sum_{j, k=1}^n \frac{\partial^2 y_j}{\partial x_k \partial x_i} \fxpartial{x_k}{y_j}.
	\]
\end{definition}

\begin{theorem}{\cite[Theorem 15]{BB24}}
	A $\lambda$-charged $\cD$-module $\mathcal{F}$ admits a structure of a $2$-differentiable $\sAV$-module of type $\bk_{\lambda}$. In a étale chart $U \subset X$ with uniformizing parameters $x_1,\dots,x_n$, the representation $\rho$ of $\Theta(U)$ is given by
	\[
	\rho \left (f\xpartial{x_i} \right ) = f\xpartial{x_i} + \lambda \fxpartial{f}{x_i}.
	\]
\end{theorem}

\begin{theorem}
	Let $\cM$ be a simple holonomic $\sAV$-module of type $W \cong W_0\otimes \bk_{\lambda}$, where $W_0$ is a $\fsl_n$-module. Then, there exists a $\lambda$-charged $\cD$-module $\mathcal{F}$ such that 
	\[
	\cM \cong \mathcal{F} \otimes_{\cO} \mathcal{J}^{W_0}.
	\]
\end{theorem}
\begin{proof}
	The proof of this theorem follows the same steps as in Theorem \ref{theorem:sheavesofavmodulesaretensorproducct} but instead of taking $Z_0$, consider $Z_{\chi}$ where $\chi$ is the kernel of the central character corresponding to the $\gl_n$-module $\bk_{\lambda}$.

\end{proof}

There exists a holonomic $\sAV$-module of type $W$ for every $\gl_n$-module $W$. In~\cite{BB24}, the authors wrote explicit transformation laws for the sheaf $\sAV$. We may use the structure sheaf $\cO$ and these transformation laws to create holonomic $\sAV$-modules using any $\gl_n$-module $W$, as the next example shows.

\begin{example}
Let $W$ be a $\gl_n$-module, $\mathcal{C} = \{U_{\alpha}\}$ be a finite étale cover of $X$ and $U \subset X$ be an étale chart with uniformizing parameters $x_1,\dots,x_n$. We will use this open cover and the structure sheaf of $X$ to create a differentiable sheaf of $A\cV$-modules $\cM$ of type $W$ on $X$. As vector space, $\cM(U) = \cO(U) \otimes W$. The action of $\cO(U)$ is given by multiplication on the left side, while the action of the Lie algebra $\Theta(U)$ on $\cM(U)$ is given by
\[
\left (f \xpartial{x_i} \right )(gw) = f\fxpartial{g}{x_i}(x) w + \sum_{k=1}^n g \fxpartial{f}{x_k}(x) (E_{ki}w).
\]
For each $\alpha$, define $\cM(U_{\alpha}) = \cO(U_{\alpha} \cap U) \otimes W$. The action of $\cO(U_{\alpha})$ on $\cM(U_{\alpha})$ is given by left multiplication, while the action of the Lie algebra $\Theta(U_{\alpha})$ on $\cM(U_{\alpha})$ is
\begin{align*}
\left (f \xpartial{y_{i}} \right )(g \otimes w) =   f \fxpartial{g}{y_{i}} \otimes  w + g\sum_{a,k,j=1}^n \left ( f\frac{\partial^2 x_j}{\partial y_{a} \partial y_{i}} \fxpartial{y_{a}}{x_k}+ \fxpartial{f}{y_{a}}  \fxpartial{y_{a}}{x_k}\fxpartial{x_j}{y_i}\right ) \otimes E_{kj}w,
\end{align*}
where $y_1,\dots,y_n$ are the uniformizing parameters of $U_{\alpha}$. The transformation law for $\cM$ comes from the transformation law for $\cO$, i.e., $f_{U_{\alpha}}\otimes w \mapsto f_{U_{\beta}}\otimes w$. The resulting sheaf is a holonomic $\sAV$-module of type $W$, and its global sections are $\cM(X) = \cO(U) \otimes W$. If $W$ is integrable, then $\mathcal{J}^W$ is a $\sAV$-submodule of $\cM$. However, $\cM$ could be simple depending on $W$ and $X$.
\end{example}
The following example applies this construction to $\mathbb{P}^1$.

\begin{example}
Set $X= \mathbb{P}^1$. Consider the usual cover of $X$ given by $U_0 = \Spec(\bk[x])$ and $U_1 = \Spec(\bk[y])$ with the gluing $x \mapsto y^{-1}$. Let $\lambda \in \bk$ and consider the $A\cV$-module $\cM(U_0) =\cO(U_0)= \bk[x]$ on $U_0$, where the action of $\cO(U_0)$ is given by multiplication and the action of $\Theta(U_0) = \Der(\bk[x])$ is given by
\[
\left (x^k \frac{d}{dx} \right ) x^l = (l+\lambda k )x^{k+l-1}, \ k,l\in \Z_+.
\]
On the other hand, $\cM(U_1) = \cO(U_0 \cap U_1) = \bk[y,y^{-1}]$ as a vector space, the action of $\cO(U_1) = \bk[y]$ is given by left multiplication and
\[
\left ( y^k \frac{d}{dy} \right) y^l = \left (l-\lambda(2-k)\right )y^{k+l-1}, \ k \in \Z_+, \ l \in \Z.
\]
The gluing in $\cM(U_0 \cap U_1) = \cO(U_0 \cap U_1)$ is given by $x \mapsto y^{-1}$. The resulting $\sAV$-module $\cM$ is simple if and only if $\lambda$ is not a half-integer.
\end{example}

Let us give an example of an indecomposable $AV$-module $M$ on an affine variety, which is a free $A$-module, decomposes as a tensor product of a $D$-module $P$ with a tensor module, $M \cong P \otimes W$, where $P$ is not free over $A$, but only projective.
\begin{example}
Let $X$ be an affine elliptic curve $y^2 = t^3 - t$. Let $A$ be the algebra of polynomial functions on $X$, $A = \bk[t,y] / \left<y^2 - t^3 + t \right>$.
Lie algebra $V = \Der A$ of vector fields on $X$ is a free $A$-module,  $V = A \tau$ \cite{BF18}, generated by vector field
\[
\tau = 2y \xpartial{t} + (3t^2 -1) \xpartial{y}.
\]

Consider an ideal $P = \left< t, y \right>$ in $A$. Since ideal $P$ is not principal, as an $A$-module, $P$ is not free, but only projective and has rank 1. Let us denote the generators of $P$ by $T$ and $Y$. We can turn $P$ into a $D$-module by introducing the following gauge fields:
\[
\tau(T) = (t^2 + 1) Y, \ \ \tau(Y) = (t^3 + t) T.
\]
It is easy to check that these are compatible with the relation 
$y Y = (t^2 -1) T$ in $P$.

Next, consider the following rank 2 AV-module $A \otimes W$, where $W$ is a 2-dimensional vector space with basis $\{v, u\}$. Fix $\alpha \in \Z$ and define the action of $V$ by
\begin{align*}
f \tau \cdot gv &= f \tau(g) v + \alpha g \tau(f) v +  g \tau(\tau(f)) u, \\
f \tau \cdot gu &= f \tau(g) u + (\alpha + 1) g \tau(f) u.
\end{align*}
It is easy to check that $A\otimes W$ is an extension of two gauge modules.

Define $AV$-module $M$ as the tensor product $M = P \otimes_A (A \otimes W) \cong P \otimes_\bk W$.
Even though $P$ is not free over $A$, the tensor product $P \otimes W$ is a free $A$-module of rank 2 with generators $t T \otimes v + Y \otimes u$, $Y \otimes v + t T \otimes u.$

\end{example}

We conclude this paper by showing that simple differentiable holonomic $\sAV$-modules restrict to locally simple modules over the tangent sheaf, provided that the $\gl_n$-module that appears in $\cM$ is not isomorphic to $\Lambda^k(\bk^n)$, the $k$-th exterior product of the natural representation $\bk^n$ of $\gl_n$. The following proof is similar to the one given to an analogous statement about tensor modules for the Lie algebra of vector fields on the $n$-dimensional torus~\cite{R96} and the affine space~\cite{LLZ18}. 

\begin{theorem}\label{theorem:avmoduleremainssimpleasavmod}
Let $W$ be a simple finite-dimensional $\gl_n$-module and let $\cM$ be a simple holonomic $\sAV$-module of type $W$. If $W$ is not isomorphic to $\Lambda^k(\bk^n)$ for some $k=0,\dots,n$, then $\cM$ is a locally simple sheaf of $\Theta$-modules. 
\end{theorem}
\begin{proof}
It is enough to prove that $\cM(U)$ is a simple $\Theta(U)$-module for each étale chart $U$. Let $U \subset X$ be an étale chart with uniformizing parameters $x_1,\dots,x_n$, $A = \cO(U)$, $\cV = \Der(A)$, and $\aD= \cD(U)$. Then, $\cM(U) \cong T(P,W)$ for some simple holonomic $\aD$-module $P$ and a simple finite-dimensional $\gl_n$-module $W$.  If $\partial_i = \xpartial{x_i}$, then
\begin{align*}
&\left (x_l^{k} \partial_i \right ) \left (fx_l^{m-k} \partial_j \right )(p \otimes w) \\
= & \left (x_l^{k} \partial_i \right ) \left ((fx_l^{m-k} \partial_j p) \otimes w + \sum_{s=1}^n \left ( (m-k) fx_l^{m-k-1} \partial_s(x_l) + x_l^{m-k}\partial_s(f) \right )p \otimes E_{sj} w \right ) \\
= & \left (\left ( fx_l^{m} \partial_i\partial_j +(m-k)f x_l^{m-1} \partial_i(x_l) \partial_j + x_l^m\partial_i(f) \partial_j \right )p \right )\otimes w \\
& + \sum_{r=1}^n  kfx_l^{m-1} \partial_r(x_l) \partial_j p \otimes E_{ri}w \\
& +(m-k) \sum_{s=1}^n  \left( \delta_{sl} f x_l^{m-1} \partial_i + \delta_{sl}x_l^{m-1} \partial_i(f) + (m-k-1)\delta_{sl}\delta_{il} f x_l^{m-2}  \right ) p \otimes E_{sj}w\\
& + \sum_{s=1}^n\left ( \left ( x_l^{m} \partial_s(f)\partial_i + (m-k) \delta_{li} x_l^{m-1} \partial_s(f) + x_l^m \partial_i(\partial_s(f)) \right ) p \right )\otimes E_{sj}w \\
& +  \sum_{r,s=1}^n \left (\left ( (m-k) k\delta_{rl} \delta_{sl} fx_l^{m-2} + \delta_{rl}k \partial_s(f) x_l^{m-1} \right ) p \right ) \otimes E_{ri}E_{sj}w. \\
\end{align*}
Therefore,
\begin{equation}\label{equation:actionofglnintensor}
\sum_{k=0}^2 (-1)^k \binom{2}{k} (x_l^k\partial_i )(fx_l^{2-k} \partial_j) (p \otimes w) = 2fp \otimes  (\delta_{il} E_{lj} - E_{li}E_{lj})w.
\end{equation}
Furthermore, 
\begin{equation}\label{equation:actionofpartials}
\partial_i ( p \otimes w ) = (\partial_i p) \otimes w.
\end{equation}
Suppose $N \subset \cM(U)$ and $v \in N$. Write
\[
v=  \sum_{s=1}^k p_s \otimes w_s.
\]
By equations~\eqref{equation:actionofglnintensor} and~\eqref{equation:actionofpartials},
\[
\sum_{i=1}^k up_q \otimes (\delta_{il} E_{lj} - E_{li}E_{lj})w_q \in N
\]
for every $u\in \aD$. By Schur's Lemma, for every $q \in P$ there exists $u_r \in \aD$ such that $u_r p_s = \delta_{rs} q$. Hence, $P \otimes (\delta_{il} E_{lj} - E_{li}E_{lj})w_s \subset N$ for each $s=1,\dots,k$. If $P \otimes w \subset N$, then $P \otimes E_{ab}w \subset N $ because 
\[
p \otimes E_{ab}w = \left ( x_{a}\xpartial{x_b} \right )( p \otimes w) - \left (x_a \xpartial{x_b}p \right ) \otimes w \in N.
\]
Therefore, the subspace $\{ w \in W \mid P \otimes w \subset N \}$ is a $\gl_n$-submodule of $W$, and we have that $(\delta_{il} E_{lj} - E_{li}E_{lj})w_q = 0$, for each $q=1,\dots,k$, since $W$ is simple and $N$ is a proper subspace of $M$. 

The term $p_s \otimes E_{a_1b_1}\cdots E_{a_kb_k} w_s $ appears in $\left ( x_{a_1}\partial_{b_k} \right )  \cdots \left ( x_{a_k}\partial_{b_k} \right )v$.
We may use the same argument given above to show that
\[
(\delta_{il} E_{lj} - E_{li}E_{lj})U(\gl_n)w_q = (\delta_{il} E_{lj} - E_{li}E_{lj}) W=0.
\]
Consequently, if $\lambda$ is the highest weight of $W$, then $\lambda( E_{ii}) \in \{0,1\}$ for each $i=1,\dots,n$. Because $W$ is a simple finite-dimensional $\gl_n$-module, $\lambda(E_{ii} - E_{i+1,i+1}) \in \Z_+$. Hence, there exists $k$ such that $\lambda(E_{ii})=1$ if $i \leq k$ and $\lambda(E_{ii})=0$ if $i>k$. We conclude that $W$ is isomorphic to $\Lambda^k(\bk^n )$, which contradicts our hypothesis. In conclusion, $M = \cM(U)$ is a simple $\Der(\cO(U))$-module. 
\end{proof}

This theorem generalizes a similar result established for gauge modules and Rudakov modules, see~\cite[Theorem 12]{BNZ21} and the next example.

\begin{example}
Suppose $X$ is affine and $W$ is an integrable $\gl_n$-module. Let $p \in X$ with associated maximal ideal $\fm_p$. Then, then $\fm_p \cV / \fm_p^2 \cV \cong \gl_n$ by~\cite[Lemma 6]{BFN19} and $W$ is a module over $\fm_p \cV$ by setting $(\fm_p^2\cV )W=0$. Furthermore, $W$ is also a module over $A$ by setting $f w = f(p)w$, where $f(p)\in \bk$ satisfy $f +\fm_p= f(p) + \fm_p$. The induced module 
\[
R_p(W) = A \# U(\cV ) \otimes_{A \# U(\fm_p)} W 
\]
is called the \emph{Rudakov module}. By~\cite[Lemma 22]{BB24}, the Rudakov module $R_p(W)$ is $2$-differentiable and the associated sheaf of $A\cV$-modules $\cM = \widetilde{R_p(W)}$ satisfy
\[
\cM(U) \cong T(\bk[\partial]\delta_p, W \otimes \bk_1),
\]
where $U\subset X$ is an étale chart and $\bk[\partial]\delta_p$ denotes the module over delta-functions on $U$ and $\bk_1$ denotes the one-dimensional $\gl_n$-module given by the trace of a matrix. Note that, by~\cite[Theorem 24]{BB24} and also as a consequence of Theorem~\ref{theorem:sheavesofavmodulesaretensorproducct}, 
\[
\cM \cong \mathcal{F}_p \otimes \mathcal{J}^{W\otimes \bk_1}
\]
if $W$ is integrable, where $\mathcal{F}_p$ denotes the $\cD$-module given by delta-functions with support at $p\in X$. By Proposition~\ref{theorem:avmoduleremainssimpleasavmod}, $R_p(W)$ is a simple $\cV$-module unless
\[
W \otimes \bk_1 \cong \Lambda^k (\bk^n),
\]
for some $k =0,\dots,n$. Since ${(\Lambda^k (\bk^n))}^{\ast}\cong \Lambda^{n-k}(\bk^n) \otimes \bk_1$, a Rudakov module remains simple as a $\cV$-module unless $W$ is isomorphic to the dual of $\Lambda^k(\bk^n)$. This was originally proved by André Zaidan in his PhD thesis~\cite[Theorem 50]{Zai20}.
\end{example}

\end{document}